\documentclass[12pt,leqno]{amsart}
\usepackage{amssymb,amsthm,amsmath,latexsym,wasysym}
\usepackage{soul}
\usepackage[all]{xy}
\usepackage{xcolor}
\usepackage{mathtools}
\usepackage{hyperref}

\newcommand{\T}{{\sf T}}
\newcommand{\Tg}{{\sf T_{\bullet}}}
\newcommand{\expg}{{\sf \exp_{\bullet}}}

\newcommand{\R}{\mathbb R}
\newcommand{\Z}{\mathbb Z}
\newcommand{\N}{\mathbb{N}}

\newcommand{\Id}{\mathrm{Id}}
\newtheorem{theorem}{\sc Theorem}[section]
\newtheorem{thm}[theorem]{\sc Theorem}
\newtheorem{lem}[theorem]{\sc Lemma}
\newtheorem{prop}[theorem]{\sc Proposition}
\newtheorem{cor}[theorem]{\sc Corollary}

\newtheorem*{con1'}{Conjecture 1'}

\theoremstyle{definition}

\newtheorem{rem}[theorem]{\sc Remark}

\begin{document}

\title[Generalized torsion elements]{Generalized torsion elements in infinite groups}
\author[Bastos]{Raimundo Bastos}
\address{Departamento de Matem\'atica, Universidade de Bras\'ilia,
Bras\'ilia-DF, 70910-900 Brazil}
\email{(Bastos) bastos@mat.unb.br }
\author[Mendon\c ca]{Luis Mendon\c ca}
\address{Departamento de Matem\'atica, Universidade Federal de Minas Gerais, Belo Horizonte-MG, 31270901 Brazil}
\email{(Mendon\c ca) luismendonca@mat.ufmg.br}


\begin{abstract}
 A group element is called generalized torsion  if a finite product of its conjugates is equal to the identity.  We show that in a finitely generated abelian-by-finite group, an element is generalized torsion if and only if its image in the abelianization is torsion. We also prove a quantitative version with sharp bounds to the generalized exponent of these groups. In particular, we provide many examples of finitely presentable torsion-free groups in which all elements are generalized torsion. We also discuss positive generalized identities in abelian-by-finite groups and related classes.
\end{abstract}

\maketitle

\section{Introduction}

An element $g$ of a group $G$ is said to be  {\it generalized torsion} if there exist $x_1, \ldots, x_k\in G$ such that \[g^{x_1} g^{x_2} \cdots g^{x_k} =1,\]
where $x^y$ denotes the conjugate $y^{-1}xy$.
We will denote by $\Tg(G)$ the set of all generalized torsion elements in $G$. The {\it generalized order} of $g \in \Tg(G)$, denoted by $o_{\bullet}(g)$, is defined to be the smallest  positive integer $k$ such $g^{x_1} g^{x_2} \cdots g^{x_k} =1$ for some $x_1,\ldots,x_k\in G$. We say that $G$ has {\it generalized exponent} $k$, writing $\exp_{\bullet}(G)=k$, if $\Tg(G) = G$  and $1 \in C^k$ for all conjugacy classes $C$ of $G$. Here $C^k = \{ c_1 \cdots c_k \mid c_i \in C\}$. We will reserve the notations $o(g)$, $\exp(G)$ and $\T(G)$ for the order of $g$, exponent of $G$ and torsion subset of $G$, respectively, all taken in the usual sense.

The existence of generalized torsion element is a well-known obstruction to bi-orderability: if $G$ has a bi-ordering $<$ then for every non-trivial element $g \in G$, $1 < g$ or $g < 1$ holds. If $1 < g$ then $1 = x^{-1}x < x^{-1}gx$ for all $x \in G$ hence products of conjugates of $x$ cannot be trivial. The case $g < 1$ is similar. So if $G$ is a bi-orderable group, then $\Tg(G) = 1$ (see \cite{Rhemtulla} for more details). In particular, residually torsion-free nilpotent groups, such as RAAGs and their subgroups, or free polynilpotent groups contain no nontrivial generalized torsion elements. In this direction it is important to emphasize that there are non-bi-orderable groups $G$ with $\Tg(G)=1$ (see \cite{Bludov} for an example). 

The subset $\Tg(G) \subseteq G$ is not in general a subgroup. It is not even closed under powers, in contrast to the usual torsion subset, although counter-examples to this seem to be rarer. In Section~\ref{sec.Adyan}, using a construction by Adyan \cite{Adyan}, we produce groups $G$ and generalized torsion elements $g \in G$ such that $\langle g \rangle$ contains central elements of infinite order, which cannot be generalized torsion.

Many authors have considered torsion-free groups with generalized torsion elements from a topological perspective (see \cite{Juan, NaylorRolfsen, Tera, MotegiTeragaito2017, Pams, ItoMotegiTeragaito2021, ItoMotegiTeragaito2023}). The first example of a torsion free group all whose elements are generalized torsion was produced by Gorchakov \cite{Gorcakov}. Gorjunski supplied then an example of a group with the same property but which is additionally finitely generated \cite{Gorjuskin}. The construction of Osin \cite{Osin} gives an example in much stronger sense: any countable torsion-free group is embeddable in a finitely generated torsion-free group having precisely two conjugacy classes. In such groups any element is conjugate to its inverse, thus producing equalities of the form $xx^g=1$, which means that Osin's groups are of generalized exponent $2$.

Away from Osin's example, one might ask about finitely presentable torsion-free groups that are generalized torsion. The group
\[ P = \langle x,y \mid x^{-1} y^2 x = y^{-2}, y^{-1} x^2 y = x^{-2}\rangle\]
has such properties, as reported in the answer to Question~3.11 in Kourovka's Notebook \cite{Kourovka}, where the argument is attributed to Churkin. Such group is known under many names such as Promislow group, Passman fours group, or the Hantzsche-Wendt group. It has attracted considerable attention recently, in connection to the negative answer by Gardam to Kaplansky's conjecture on units in group algebras \cite{Gardam}. The Promislow group can be written as an extension of a free abelian group of rank $3$, generated by $x^2$, $y^2$ and $(xy)^2$, by $C_2 \times C_2$, that is, $P$ is abelian-by-finite. For this class of groups, we prove the following result.

\begin{theorem}  \label{thm.abelianbyfinite}
Let $G$ be a finitely generated abelian-by-finite group. Then $\Tg(G) = \pi^{-1}( \T(G^{ab}))$, where $\pi \colon G \to  G^{ab}$ is the abelianization map and $\T(G^{ab})$ is the torsion subgroup of $G^{ab}$. In particular:
\begin{enumerate}
    \item $\Tg(G)$ is a normal subgroup of $G$,
    \item $\Tg(G) = G$ if and only if $G^{ab}$ is finite,
    \item Additionally, if $G$ is torsion-free, then $\Tg(G) = G$ if and only if  $Z(G)=1$. 
\end{enumerate}
\end{theorem}

The result above applies in quite different situations. For instance, the generalized torsion subset of the Klein bottle group $\langle x, y \mid x^y = x^{-1} \rangle$
is the subgroup generated by $x$. For the infinite dihedral group $D_\infty$ item (2) applies and we find that $\Tg( D_\infty)  = D_\infty$. 

The groups as in item (3) above are precisely the fundamental groups of flat manifolds with trivial first Betti number. These are actually Bieberbach groups and their holonomy groups are \emph{primitive} in the sense of Hiller and Sah \cite{HillerSah}. Recall that a finite group is primitive if none of its Sylow subgroups that are cyclic admit a normal complement. As the class of primitive groups is quite vast, we find a correspondingly large family of Bieberbach groups of finite generalized exponent; see Corollary~\ref{cor:Bieber}.

The following result is a quantitative version of Theorem~\ref{thm.abelianbyfinite}.

\begin{theorem}  \label{thm.bounds}
Let $G$ be a finitely generated abelian-by-finite group, and let $A \leq G$ be a free abelian normal subgroup of finite index. If $G^{ab}$ is finite, then $\exp(G^{ab}) \leq \expg(G) \leq |G/A|$.
\end{theorem}

Both bounds are attained for the Promislow group $P$. More generally, the bounds are sharp for the Gupta--Sidki metabelian groups $K(p)$ \cite{GuptaSidki1999}. These groups are defined for all primes $p$ as generalizations of the Promislow group, which is isomorphic to $K(2)$. A further generalization is given by the groups $K(p^n,p^m)$ defined by Cid \cite{Cid}. Here $p$ may be any prime and $n$ and $m$ any positive integers. With $n=m=1$ one has $K(p^1,p^1) \simeq K(p)$. All these groups are Bieberbach groups, so we may take the subgroup $A$ in the theorem to be their translation subgroups. We discuss Bieberbach groups and these examples in Section~\ref{sec.ab.by.finite}.

\begin{cor} Let $p$ be a prime. Then \[\expg(K(p^n,p^m)) = p^{n+m}.\] In particular, $\expg(P) = 4$ and $\expg(K(p)) = p^2$.
 \end{cor}

It is not clear which integers occur as the generalized exponent of a finitely presented torsion-free group. For partial results, see Proposition~\ref{powerful}, where in particular we show that any square occurs in that situation.

In a similar direction, we say that a group $G$ satisfies a \emph{positive generalized identity} of degree $n$ if there exist fixed elements $x_1, \ldots, x_n \in G$ such that 
\[g^{x_1} g^{x_2} \cdots g^{x_n} =1\]
for all $g \in G$. Following Endimioni \cite{Endimioni2006}, we denote by  $\widehat{\mathcal{B}}_n$ the class of all groups satisfying such an identity.

We interpret our results for abelian-by-finite groups in terms of positive identities in Section~\ref{sec.positive.id}. The proof of Theorem~\ref{thm.abelianbyfinite} actually shows that an abelian-by-finite group satisfies a positive generalized identity if and only if it is of finite generalized exponent; see Corollary~\ref{cor.strongexp} and Theorem~\ref{BHat}.

Endimioni's results imply that groups in $\widehat{\mathcal{B}}_n$ have quite restricted structure for small values of $n$. For instance, all finitely generated groups in $\widehat{\mathcal{B}}_4$ are abelian-by-finite. On the other hand, an example  by Casolo (see \cite[Section~5]{Endimioni2006}) shows that there is a finitely generated solvable group in  $\widehat{\mathcal{B}}_8$ that is not polycyclic. Building on this example, we produce a family of finitely generated torsion-free groups which satisfy positive generalized identities, but are not polycyclic. Additionally, these examples can be chosen to be solvable of arbitrarily large derived length.

\section{Preliminaries}

In the next result we collect some of the basic properties of generalized torsion elements (see \cite[Proposition 2.2]{BSS} for more details).

\begin{prop} \label{prop_basic}
 Let $G$ be a group and let $g \in G$. 
\begin{enumerate}
    \item If $K$ is a group and $\varphi \colon G \to K$ is a homomorphism, then $\varphi(\Tg(G)) \subseteq \Tg(K)$. 
   \item If $x  \in \Tg(G) \cap Z(G)$, then $x \in \T (G)$. In particular, if $G$ is abelian then $\Tg(G) = \T (G)$. 

    \item The order of $gG'$ in $G^{ab}$ divides the generalized order of $g$ in $G$.   
\end{enumerate} 
\end{prop}

It is not true in general that if $g \in \Tg(G)$, then $g^n \in \Tg(G)$ (see Proposition \ref{prop:Adyan}, below). In the opposite direction, one has the following.

\begin{lem} \label{powers}
If $g^{y_1} \cdots g^{y_n} \in \Tg(G)$, then $g \in \Tg(G)$. In particular, if $g^n \in \Tg(G)$ for some $n$, then $g \in \Tg(G)$.
\end{lem}
\begin{proof}
By definition, there exist $x_1, \ldots, x_k \in G$  such that \[1 = (g^{y_1} \cdots g^{y_n} )^{x_1} \cdots (g^{y_1} \cdots g^{y_n} )^{x_k} = g^{y_1 x_1} \cdots g^{y_n x_1} g^{y_1x_2} \cdots g^{y_n x_k},\]
so $g$ is a generalized torsion element.
\end{proof}

\begin{lem} \label{extension}
Any extension of generalized torsion groups is generalized torsion.
\end{lem}
\begin{proof}
Let $Q = G/N$, where both $N$ and $Q$ are generalized torsion. For arbitrary $g \in G$, we may find $x_1, \ldots, x_n \in G$ such that 
\[ (gN)^{x_1N} \cdots (gN)^{x_nN} = 1,\]
so $g^{x_1} \cdots g^{x_n} \in N$, and by hypothesis it is a generalized torsion element. Applying Lemma~\ref{powers}, we find that $g$ is generalized torsion.
\end{proof}

The proof above also says that if $\expg(N) = n$ and $\expg(Q)=q$, then $\expg(G) \leq nq$.

\subsection{Adyan's groups}
\label{sec.Adyan}
The set of all generalized torsion elements in a group is not well-behaved. We will present here an example related to the Burnside problem, which shows that the converse of Lemma~\ref{powers} does not hold.  For integers $m \geq 2$ and $n \geq 665$ odd, Adyan constructed in \cite{Adyan} a torsion-free group $A(m,n)$ with the property that  the intersection of any two non-trivial subgroups is non-trivial. Such group is a central extension of the form
\[ 1 \to \Z \to A(m,n) \to B(m,n) \to 1,\]
where $B(m,n)$ is the free Burnside group on $m$ generators and exponent $n$.

\begin{prop} \label{prop:Adyan}
Let $G = A(m,n)$ with $m \geq 2$ and $n \geq 665$ odd.
\begin{enumerate}
    \item If $1 \neq \alpha \in \Tg(G)$, then $\alpha^n \not\in \Tg(G)$. In particular, $\Tg(G)$ need not be a subgroup.
    \item The set of all commutators, $\{[x,y] \mid x,y \in G\}$, is contained in the set of generalized torsion elements.
\end{enumerate}
\end{prop}

\begin{proof}
Denote by $N$ the central subgroup with $G/N = B(m,n)$.
\begin{enumerate}
    \item It is sufficient to prove that $\alpha^n \not\in \Tg(G)$. Since $G$ is torsion-free, we deduce that $\alpha^n \neq 1$. Moreover, by construction, $1 \neq \alpha^n \in Z(G)$ and this element cannot be a generalized torsion element (by Proposition~\ref{prop_basic}~(2)). 
    \item Choose arbitrary elements \(x, y \in G\). If \(x\) and \(y\) commute, then \(1 = [x,y] \in \Tg(G)\). Now, consider two elements \(x, y \in G\) with \([x,y] \neq 1\). Since \(x^n \in N \subseteq Z(G)\), it follows that \( [x^n, y] = 1 \), and thus, \[[x^n,y] = [x,y]^{x^{n-1}} [x,y]^{x^{n-2}} \cdots [x,y]^x [x,y]=1.\]
Consequently, $[x,y]$ is a generalized torsion element. \qedhere
\end{enumerate}
\end{proof}

\section{Abelian-by-finite groups} \label{sec.ab.by.finite}

The following is an immediate consequence of the proof of \cite[Theorem~22]{Farkas}. We supply the proof for the convenience of the reader.

\begin{lem} \label{transfer}
Let $G$ be a finitely generated abelian-by-finite group. Assume that $A$ is a torsion-free abelian subgroup of finite index. If $1 \neq z \in Z(G) \cap A$, then $zG'$ has infinite order in $G^{ab}$.
\end{lem}
\begin{proof}
It is enough to find a group homomorphism $\tau \colon G \to B$, where $B$ is an abelian group and $\tau(z)$ has infinite order.

Set $n=|G:A|$. Let $\tau \colon G \to A$ be the transfer homomorphism associated to the identity $\Id_A \colon A \to A$. According to \cite[10.1.2]{Rob}, if $z \in Z(G)$, then $\tau(z)=z^n$, which is of infinite order if $1 \neq z \in A$.
\end{proof}

\begin{theorem}  \label{Thm.ab.by.finite}
Let $G$ be a finitely generated abelian-by-finite group. Then $\Tg(G) = \pi^{-1}( \T(G^{ab}))$, where $\pi \colon G \to  G^{ab}$ is the abelianization map and $\T(G^{ab})$ is the torsion subgroup of $G^{ab}$. 
\end{theorem}
\begin{proof}
It is clear that if $g \in \Tg(G)$, then its image in $G^{ab}$ is torsion. 

For the converse, let $A \leq G$ be a free abelian normal subgroup of finite index of $G$. Let $g \in G$ and assume that $\pi(g)$ has finite order. For some $k\geq 1$ we have $g^k \in A$, and if $g^k$ is generalized torsion, then so is $g$ by Lemma~\ref{powers}. Moreover $\pi(g^k)$ has finite order too, so we may assume without loss of generality that $g \in A$. 

Let $1, x_2, \ldots, x_n$ be a set of representatives of the cosets of $A$ in $G$. Then $z = gg^{x_2} \cdots g^{x_n}$ is a central element of $G$ and $\pi(z)= \pi(g)^n$ has finite order. By Lemma~\ref{transfer}, $z = 1$. So $g$ is generalized torsion.
\end{proof}

In the case where $G^{ab}$ is finite, the proof above shows that $G$ is of generalized exponent at most $\exp(G/A) \cdot |G/A|$. A more careful analysis leads to the following improvement.

\begin{prop}  \label{upperbound.genexp}
Let $G$ be a finitely generated abelian-by-finite group. Assume that $A$ is a torsion-free abelian normal subgroup of finite index. If $G^{ab}$ is finite, then $G$ has generalized exponent at most $|G/A|$.
 \end{prop}

\begin{proof}
Let $g \in G$ and let $n$ be the order of $gA \in G/A$. Let $S \subset G$ be a transversal for $A \cdot \langle g \rangle$ in $G$.  We will show that the element
\[ z = \prod_{s \in S} (g^n)^s\]
is trivial. Notice that $z$ is a product of $n \cdot |S| = |G/A|$ conjugates of $g$, which will prove the claim.

It will be enough to show that $z^n = 1$, since $z \in A$ and $A$ is torsion-free. We have:
\[ z^n = \prod_{s \in S} (g^{n^2})^s = \prod_{s \in S} \left( \prod_{i=0}^{n-1}(g^n)^{g^i}\right)^s = \prod_{t \in T} (g^n)^t,\]
where $T = \{g^i s \mid  0\leq i < n, s \in S\}$. But $g^n \in A$ and $T$ is a transversal for $A$ in $G$, so $\prod_{t \in T} (g^n)^t = 1$ by the argument in the proof of Theorem~\ref{Thm.ab.by.finite}. 
\end{proof}

We will see that the bound is sharp for the Promislow group and its generalizations in the sequence.

\subsection{Examples} \label{examples.abelianbyfinite}
In this Section we will provide some examples of applications of Theorem~\ref{Thm.ab.by.finite}.

\subsubsection{Crystallographic groups}

Let $G$ be an $n$-dimensional crystallographic group, i.e. a discrete and cocompact subgroup of $\operatorname{Iso}(\R^n) = O(n) \ltimes \R^n$, the group of isometries of an $n$-dimensional euclidean space. By the first Bieberbach theorem, $G$ fits into an exact sequence
\[1 \to A \to G \to Q \to 1, \]
where $Q$ is a finite group and $A$ is a free abelian group of rank $n$ and it is the unique maximal abelian (normal) subgroup of $G$.  The groups $Q$ and $A$ are called the \emph{holonomy group} and the \emph{translation group} of $G$, respectively. When $G$ is additionally torsion-free, it is called a \emph{Bieberbach group}.

All crystallographic groups fall into the scope of Theorem~\ref{Thm.ab.by.finite}. The most interesting cases for us are the  Bieberbach groups of trivial center. Equivalently, these are the Bieberbach groups with finite abelianization (see \cite[Prop.~1.4]{HillerSah}). 

\begin{cor} \label{cor:Bieber}
There are Bieberbach groups $G$ of finite generalized exponent satisfying any of the following properties:
\begin{enumerate}
    \item $G$ is solvable of derived length $l$, for any $l \geq 2$;
    \item $G$ is of arbitrarily large generalized exponent;
    \item $G$ is non-solvable.
\end{enumerate}
\end{cor}
\begin{proof}
By Hiller and Sah's classification, it enough to find primitive groups  $Q$ with the appropriate property. Going case by case, we choose $Q$ to be:
\begin{enumerate}
    \item a $p$-group of derived length $l-1$;
    \item $C_n \times C_n$. Notice that if $G$ is a Bieberbach group with holonomy group $Q = C_n \times C_n$, then its generalized exponent is at least $n$ by Proposition~\ref{prop_basic}(3);
    \item a non-abelian simple group. \qedhere
\end{enumerate}
\end{proof}

\subsubsection{Promislow group and its generalizations}  \label{PromislowGuptaSidki}
The Promislow group is defined by the presentation
\[ P = \langle x, y \mid (x^2)^y = x^{-2}, (y^2)^x = y^{-2}\rangle.\]
The normal subgroup $A = \langle x^2, y^2, (xy)^2 \rangle$ is free abelian of rank $3$, and the quotient $P/A$ is isomorphic to $C_2 \times C_2$, while $P/P' \simeq C_4 \times C_4$.

In order to define its generalizations, we will use the following notation.  For $f(T) = n_0 + n_1 T + \ldots + n_k T^k$ a polynomial with integer coefficients and $g$ and $h$ group elements, we denote
\[ g^{f(h)} = (g^{n_0})(g^{n_1})^h \cdots (g^{n_k})^{h^k}.\] 
For $p$ a prime and $m,n \in \N$, let
\[ K(p^n, p^m) = \langle x, y \mid [[x,y], x^{p^n}], [ [x,y],y^{p^m}], (x^{p^n})^{\Psi_{p^m}(y)}, (y^{p^m})^{\Psi_{p^n}(x)} \rangle_{\text{metab}},\] 
where $\Psi_k(T) = 1 + T + \ldots + T^{k-1}$ and \emph{metab} indicates a presentation in the variety of metabelian groups. These groups were
defined by Cid \cite{Cid}. The cases $K(p) \coloneqq K(p^1,p^1)$ is the family of groups considered by Gupta and Sidki \cite{GuptaSidki1999}. The Promislow group coincides with $K(2) = K(2,2)$. 

It is proved in \cite{GuptaSidki1999, Cid} that each $K(p^n,p^m)$ is a Bieberbach group with holonomy $C_{p^n} \times C_{p^m}$ and abelianization isomorphic to $C_{p^{m+n}} \times C_{p^{m+n}}$.

\begin{prop}  \label{expg.Cid}
Cid group $K(p^n,p^m)$ has generalized exponent $p^{m+n}$. In particular, the Promislow group has generalized exponent $4$, and the Gupta--Sidki group $K(p)$ has generalized exponent $p^2$.
\end{prop}
\begin{proof}
Since the holonomy of $K(p^n,p^m)$ is $C_{p^n} \times C_{p^m}$, it follows by Proposition~\ref{upperbound.genexp} that its generalized exponent is at most $p^{m+n}$. On the other hand, the abelianization $K(p^n,p^m)^{ab}$ has elements of order $p^{m+n}$, so Proposition \ref{prop_basic}(3) gives that $\expg(K(p^n,p^m)) \geq p^{m+n}$.
\end{proof}

Recall that a positive integer $n$ is \emph{powerful} if for any prime $p$, one has $p\mid n$ implies $p^2 \mid n$.

\begin{prop}  \label{powerful}
If $n$ is powerful, then there exists a Bieberbach group $G$ with generalized exponent $n$.    
\end{prop}
\begin{proof}
Let $n = \prod p_i^{k_i}$ be its prime factorization. For each $i$, let $G_i=K(p_i,p_i^{k_i-1})$, and let $G = \prod G_i$. 

By Proposition~\ref{expg.Cid}, each $G_i$ has generalized exponent $p_i^{k_i}$, so $G$ has generalized exponent at most $\prod p_i^{k_i} = n$ by Lemma~\ref{extension}. Similarly, since $G_i^{ab}$ has exponent $p_i^{k_i}$, the abelianization of $G$ has exponent  $\prod p_i^{k_i} = n$. So $\expg(G) = n$.
\end{proof}

\begin{rem}
We sketch an argument to prove that non-trivial Bieberbach groups cannot have generalized exponent $2$. Suppose that $G$ is such a group and let $A$ be its translation subgroup. Then for each $a \in A$, there exists $g \in G$ such that $a^g = a^{-1}$. By \cite[Proposition~2.1]{HillerSah}, the action of $\langle g \rangle$ on $A$ must have a non-trivial fixed-point. In particular, there is no $g \in G$ such that $a^g  = a^{-1}$ for all $a \in A$. For $g \in G$, put $A_g = \{a \in A \mid a^g = a^{-1}\} \leq A$. Then $A = \cup A_g$, where $g$ runs through a transversal of $A$ in $G$. By Neumann's theorem \cite{Neumann}, this is impossible unless some $A_g$ is of finite index in $A$, which is not the case. 
\end{rem}

\subsubsection{Free abelianized extensions}
For $Q$ a finite group, written as a quotient $Q = F/R$ of a free group of finite rank, we can consider the free abelianized extension
\[ 1 \to R/[R,R] \to F/[R,R] \to Q \to 1.\]
It is an extension of $Q$ by the free abelian group $R/[R,R]$, which is of rank $|Q|(\mathrm{rk}(F) - 1) + 1$. It is well known that $G = F/[R,R]$ is torsion-free. The abelianization $G^{ab} = F/[F,F]$ is free abelian, so $\Tg(G) = G'$ by  Theorem~\ref{Thm.ab.by.finite}.

\subsubsection{Wreath products}
A wreath product of the form $G = \Z \wr Q = \Z Q \rtimes Q$ with $Q$ finite has abelianization $\Z \times Q^{ab}$, and the abelianization map is $\pi (m,q) = (\epsilon(m), qQ')$, where $\epsilon \colon \Z Q \to \Z$ is the augmentation map. So $\Tg(G) = Aug(\Z Q) \rtimes Q$.
 
\section{Positive generalized identities}  \label{sec.positive.id}
Let $n$ be a positive integer. Following Endimioni \cite{Endimioni2006}, we denote by $\widehat{\mathcal{B}}_n$ the class of groups $G$  satisfying an identity of the form
\begin{equation} \label{positive.id} g^{x_1} \cdots g^{x_n}=1, \  \ \forall g \in G,\end{equation}
where $x_1, \ldots, x_n$ are \emph{fixed} elements of $G$, and we say that \eqref{positive.id} is a positive generalized identity in $G$. Groups in $\widehat{\mathcal{B}}_n$ are clearly of finite generalized exponent. Note also that $\widehat{\mathcal{B}}_n$ contains the Burnside variety $n$, denoted here by $\mathcal{B}_n$.

\subsection{Positive generalized identities in abelian-by-finite groups}
Note that the argument of Churkin in \cite{Kourovka} already shows that the Promislow group $P$ satisfies the positive generalized identity
\[ g^2 (g^2)^x (g^2)^y  (g^2)^{xy} = 1\]
of degree $8$. In the same direction, the proof of Theorem~\ref{Thm.ab.by.finite} actually shows that an abelian-by-finite group satisfies a positive generalized identity if and only if it is of finite generalized exponent. The bounds are as follows.

\begin{cor} \label{cor.strongexp}
Let $G$ be a finitely generated abelian-by-finite group. Assume that $A$ is a torsion-free abelian normal subgroup of finite index. If $G^{ab}$ is finite, then $G \in \widehat{\mathcal{B}}_n$ where $n =\exp(G/A)\cdot |G/A|$.
\end{cor}
\begin{proof}
Arguing as in the proof of Proposition~\ref{upperbound.genexp}, we see that for $g \in G$ we have
\[ (g^k)^{x_1} (g^k)^{x_2} \cdots (g^k)^{x_m}=1,\]
where $k = \exp(G/A)$ and $\{x_1, \ldots, x_m\}$ is a transversal for $A$ in $G$.
\end{proof}

In \cite{Endimioni2006}, Endimioni studied the classes  $\widehat{\mathcal{B}}_n$ for small values of $n$. He showed that  $\mathcal{B}_n$ and $\widehat{\mathcal{B}}_n$ coincide for $n \leq 2$ and $\mathcal{B}_3 \subsetneq \widehat{\mathcal{B}}_3  \subsetneq \mathcal{B}_9$, and that finitely generated groups in $\widehat{\mathcal{B}}_4$ are abelian-by-finite, but cannot be torsion-free.
Moreover, he observes that the Promislow group belongs to $\widehat{\mathcal{B}}_8$. The case of prime degree was studied in \cite{ADFG}, where it it is shown that any finitely generated solvable group in $\widehat{\mathcal{B}}_p$ is a finite $p$-group. We obtain the following related result.

\begin{thm}  \label{BHat}
Let $n$ be a positive integer. If $p^3$ divides $n$ for some prime $p$, then $\widehat{\mathcal{B}}_{n}$ contain torsion-free groups. In particular, $\mathcal{B}_{n}  \subsetneq \widehat{\mathcal{B}}_{n}$.
\end{thm}

\begin{proof}
Since $p^3$ is a divisor of $n$, it follows that $\widehat{\mathcal{B}}_{p^3} \subseteq \widehat{\mathcal{B}}_{n}$. So, it is sufficient to prove that $\widehat{\mathcal{B}}_{p^3}$ contain torsion-free groups. Set $G=K(p)$. By Corollary \ref{cor.strongexp}, $G \in\widehat{\mathcal{B}}_{p^3}$, which completes the proof.
\end{proof}

It is worth to mention that, by \cite{Endimioni2006}, there are no torsion-free groups in $\widehat{\mathcal{B}}_{4}$. In \cite{ADFG}, the authors provides finitely generated metabelian groups in $\widehat{\mathcal{B}}_{n^2}$, for every positive integer $n$. However, these groups are not torsion-free. In this direction, it seems natural to ask whether there are torsion-free groups in $\widehat{\mathcal{B}}_{p^k}$, for  $p$ odd and $k\leqslant 2$.

\subsection{Torsion-free groups related to Casolo's example}

It is clear that if $G$ is a finitely generated metabelian group in some $\widehat{\mathcal{B}}_{n}$, then $G$ is polycyclic. The same is not true for solvable groups with derived length $d\geq 3$, as shown in the example by Casolo in \cite[Section~5]{Endimioni2006}. His example is a semidirect product of $\Z D_{\infty}$ by $D_{\infty} \times C_2$, which contains torsion. We will provide a family of generalizations, producing finitely generated torsion-free groups which are not polycyclic but belong to some $\widehat{\mathcal{B}}_{n}$.

We start with any finitely generated torsion-free group $G$ in some $\widehat{\mathcal{B}}_{m}$  and such that $|G^{ab}|$ is even (e.g. the Promislow group). We form the group
\[ \Gamma_G  = (\Z G) \rtimes (G \times G)\] 
where the first copy of $G$ acts as product on the left and the second copy acts by inversion (via some homomorphism $\pi \colon G \to C_2$).

If $g^{x_1} \cdots g^{x_m}=1$ is a positive generalized identity in $G$, then there exist elements $y_1, \ldots, y_m \in \Gamma_G$ such that $\gamma^{y_1} \cdots \gamma^{y_m} \in \Z G \subseteq \Gamma_G$ for all $\gamma \in \Gamma_G$. By choosing $\sigma \in G \times G$ that acts on $\Z G$ by inversion, we find that 
\[  \gamma^{y_1} \cdots \gamma^{y_m} \gamma^{y_1 \sigma } \cdots \gamma^{y_m \sigma} = 1 \]
is a positive generalized identity in $\Gamma_G$. So we produced a finitely generated torsion-free group satisfying a positive generalized identity and containing a copy of $\Z \wr G$. 

We may apply this construction to build a sequence $\{G_n\}_{n=1}^{\infty}$ where $G_1 = P$ the Prosmislow group and $G_{n+1} = \Gamma_{G_n}$ (for $n \geq 1$). This produces finitely generated torsion-free groups satisfying positive generalized identities and whose derived lengths increase with $n$. Moreover, $G_n$ is not polycyclic for $n \geq 2$; compare with \cite[Corollary~4.1]{Endimioni2006}.

\section*{Acknowledgements}
We are grateful to Sam Corson for useful discussions regarding orderable groups. This work was partially supported by  ``Conselho
Nacional de Desenvolvimento Cient\'ifico e Tecnol\'ogico – CNPq'', FAPDF and FAPEMIG [APQ-02750-2].

\end{document}